\documentclass[11pt]{amsart}
\usepackage{amsmath,amssymb,mathrsfs,enumerate}
\usepackage[bookmarks=true]{hyperref}
\newtheorem{theorem}{Theorem}[section]
\newtheorem{proposition}[theorem]{Proposition}
\newtheorem{corollary}[theorem]{Corollary}

\newtheorem{remark}[theorem]{Remark}
\newtheorem{lemma}[theorem]{Lemma}

\theoremstyle{definition}

\newcommand{\cF}{\mathcal{F}}

\newcommand{\bR}{\mathbb{R}}

\newcommand{\p}{\partial}
\newcommand{\gd}{\nabla}
\newcommand{\lp}{\Delta}

\DeclareMathOperator{\vol}{vol}
\newcommand{\nc}{\newcommand}
\nc{\on}{\operatorname}
\nc{\ve}{\varepsilon}
\nc{\area}{\on{Area}}
\nc{\tr}{\on{{tr}}}
\nc{\la}{\langle}
\nc{\rg}{\rangle}
\nc{\ric}{\on{Ric}}

\begin{document}
\title{Rigidity Results Involving Stabilized Scalar Curvature}
\author{Yipeng Wang}
\address{Columbia University \\ 2990 Broadway \\ New York, NY 10027 \\ USA}
\begin{abstract}
    We establish a rigidity theorem for Brendle and Hung's recent systolic inequality, which involves Gromov's notion of \(T^{\rtimes}\)-stabilized scalar curvature. Our primary technique is the construction of foliations by free boundary weighted constant mean curvature hypersurfaces, enabling us to generalize several classical scalar curvature rigidity results to the \(T^{\rtimes}\)-stabilized setting. Additionally, we develop a monotone quantity using Ricci flow coupled with a heat equation, which is essential for rigidity analysis.
\end{abstract}

\maketitle
\section{Introduction}
In this article, we explore several rigidity questions arising in the context of scalar curvature. Our main result is a rigidity theorem that builds upon the recent systolic inequality established by Brendle and Hung \cite{Brendle-Hung}. This inequality applies to compact manifolds with nonnegative scalar curvature in a suitably weak sense.
\begin{theorem}{\label{thm:free-boundary}}
    Let $(M^n, g)$ be a compact, connected, orientable Riemannian manifold of dimension $n \leq 7$ with nonempty boundary $\partial M$. Denote by $\eta$ the outward-pointing unit normal vector along $\partial M$. Let $\varphi$ be a smooth function defined on $M$. Suppose there is a smooth map $\xi: \partial M \to S^1$, as well as a smooth map $(\theta_1, \ldots, \theta_{n-2}): M \to T^{n-2}$, such that $
(\xi,\theta_1,\cdots,\theta_{n-2}):\p M\to S^1\times T^{n-2}$
has non-zero degree. Let $\Xi$ be the pullback of the volume form on $S^1$ via $\xi$, and let $\Theta_i$ be the pullback of the volume form on $S^1$ via $\theta_i$. Define
\[
\sigma(M,g):=\inf\{\mathcal{H}^1(\alpha):\alpha\text{ is a smooth, closed curve on }\partial M\text{ with }\int_{\alpha}\Xi\ne 0\}.
\]
Assume that
    \[
    -2\Delta_M\varphi-|\nabla^M\varphi|^2+R_M\geq 0
    \]
    at point point in $M$. Then we have
    \[
    \inf_{\partial M}(H_{\partial M} + \langle \nabla^M \varphi, \eta \rangle) \cdot \sigma(M,g) \leq 2\pi.
    \]
    Moreover, equality occurs if and only if $\varphi$ is constant on $M$ and the universal cover of $M$ is isometric to $B^2(r)\times \mathbb{R}^{n-2}$ with the standard product metric for some $r>0$.
\end{theorem}
The inequality in Theorem \ref{thm:free-boundary} was recently proved by Brendle--Hung \cite{Brendle-Hung}, extending prior results by Lucas Ambrozio \cite{Ambrozio} in the three-dimensional setting. An analogue of Theorem \ref{thm:free-boundary} for closed manifolds is the systolic inequality established by Zhu \cite{Zhu}, which generalizes the rigidity theorem of Bray--Brendle--Neves \cite{BBN} in dimension three. For additional scalar curvature rigidity results involving minimal surfaces in 3-manifolds, we refer to \cite{Cai-Galloway}, \cite{Micallef-Moraru}, and \cite{Nunes}.

Brendle and Hung’s approach relies on the minimal slicing argument developed by Schoen and Yau \cite{SY-3manifold}. This technique is closely related to the torical symmetrization methods discussed in \cite{GL-1}, \cite{Fischer-Colbrie-Schoen}, and \cite{Gromov-GAFA}.

The quantity $-2\Delta_M\varphi-|\nabla^M\varphi|^2+R_M$ appearing in Theorem \ref{thm:free-boundary} is closely related to the notion of $T^{\rtimes}-$stabilized scalar curvature introduced by Gromov \cite{Stabilized}. The idea is to perform an ``infinite dimensional torical symmetrization" of $M$. Let $N$ be a positive integer, and consider $\tilde{M}=M\times T^N$, where $T^N$ is an $N$-dimensional torus  equipped with a flat metric $g_{T^N}$. Define a warped product metric on $\tilde{M}$ by 
$$\tilde{g}=g+e^{\frac{2}{N}\varphi}g_{T^N}.$$ Under this construction, the scalar curvatures of $(\tilde{M},\tilde{g})$ and $(M,g)$ are related by
\[
R_{\tilde{M}}=R_M-2\Delta_M\varphi-\frac{N+1}{N}|\nabla^M\varphi|^2.
\]
As $N\to\infty$, , one can heuristically interpret $-2\Delta_M\varphi-|\nabla^M\varphi|^2+R_M$ as the scalar curvature of an ``infinite-dimensional" Riemannian manifold. Moreover, if we take $\varphi$ to be a constant, this expression reduces to the standard scalar curvature on $M$. Hence, the assumption in Theorem \ref{thm:free-boundary} can be viewed as a weaker notion of nonnegative scalar curvature on $M$.

In the literature, the quantity \( -2\lp_M \varphi - |\nabla^M \varphi|^2 + R_M \) is also known as the weighted scalar curvature or the \(P\)-scalar curvature. This terminology arises because this quantity can be viewed as the scalar curvature of the weighted manifold \((M, g, e^{\varphi} \, d\vol_g)\). A crucial application of this concept appears in the study of Ricci flow \cite{perelman2002entropyformularicciflow}. In this setting, the corresponding Hilbert-Einstein action on the weighted manifold is referred to as Perelman's \(\cF\)-functional:
\[
    \cF(M, g, -\varphi) := \int_M \left( -2\lp_M \varphi - |\nabla^M \varphi|^2 + R_M \right) e^{\varphi} \, d\vol_g.
\]
Perelman observed that along the Ricci flow, if \(\varphi\) satisfies a backward parabolic equation, then the \(\cF\)-functional is monotonically increasing (see Remark \ref{rmk:Perelman}).

A similar argument to that used in the proof of Theorem \ref{thm:free-boundary} for closed manifolds also extends Zhu's result \cite{Zhu} to the context of stabilized scalar curvature, as follows:
\begin{theorem}{\label{thm:Zhu}}
Let $(M^n,g)$ be a closed, connected, orientable Riemannian manifold with $n\le 7$, and $\varphi$ be a smooth function on $M$. Suppose there is a smooth map $\omega:M\to S^2$, as well as a smooth map $(\theta_1,\cdots,\theta_{n-2}):M\to T^{n-2}$, such that $(\omega,\theta_1,\cdots,\theta_{n-2}):M\to S^2\times T^{n-2}$ has non-zero degree. Let $\Omega$ be the pull-back of the volume form on $S^2$ via $\omega$, and let $\Theta_i$ be the pullback of the volume form on $S^1$ via $\theta_i$. Define
$$\mathscr{A}(M,g):=\inf\{\mathcal{H}^2(\Sigma):\Sigma\text{ is a smooth surface of }M\text{ with }\int_{\Sigma}\Omega\ne 0\}.$$
    Then we have
\[
\inf_M(-2\Delta_M\varphi-|\nabla^M\varphi|^2+R_M)\cdot\mathscr{A}(M,g)\le 8\pi.
\]
Moreover, equality holds if and only if $\varphi$ is constant on $M$ and the universal cover of $M$ is isometric to $S^2(r)\times \mathbb{R}^{n-2}$ with the standard product metric for some $r>0$.
\end{theorem}
The positive mass theorem and its rigidity in the context of stabilized scalar curvature have been recently studied by Baldauf and Ozuch \cite{Baldauf-Ozuch} and Law, Lopez, and Santiago \cite{Law-Lopez-Santiago} using the Dirac operator method, as well as by Chu and Zhu \cite{Chu-Zhu} employing a weighted minimal slicing argument. It would also be interesting to investigate whether other classical rigidity results in scalar curvature can be generalized to this weaker setting. One of the earliest significant results concerning scalar curvature rigidity trace back to the work of Gromov and Lawson, as well as Schoen and Yau, on the Geroch conjecture.

\begin{theorem}[Gromov--Lawson \cite{GL-1}, \cite{GL-2}; Schoen--Yau \cite{SY-3manifold}]\label{thm:Geroch} Let $g$ be a Riemannian metric on the torus $T^n$. Then $\inf_{T^n} R_g \leq 0$. Moreover, equality holds if and only if $g$ is flat. \end{theorem}

The same technique also yields a generalization of Theorem \ref{thm:Geroch} to the stabilized scalar curvature setting. The inequality was established in the work of Chu and Zhu \cite{Chu-Zhu}. Here, we also provide a proof of the rigidity statement.

\begin{theorem}{\label{thm:generalized-Geroch}}
Let $(M^n,g)$ be a closed, connected, orientable Riemannian manifold with $n\le 7$, and let $\varphi$ be a smooth function on $M$. Suppose there exists a smooth map $M\to T^{n}$ of non-zero degree. Then
    \[
    \inf_M (-2\Delta_M\varphi-|\nabla^M\varphi|^2+R_M)\le 0.
    \]
    Moreover, equality holds if and only if the universal cover of $M$ is isometric to $\mathbb{R}^n$ with the flat metric, and $\varphi$ is constant on $M$.
\end{theorem}
\begin{remark}
    Recent work by Brendle, Hirsch, and Johne \cite{intermediate} provides a generalization of the Geroch conjecture in the setting of intermediate curvature. The corresponding rigidity problem has been studied by Chu, Kwong, and Lee \cite{intermediate-rigidity}. It would be intriguing to consider analogues of stabilized curvature for other curvature conditions as well.
\end{remark}

The dimension constraint in Theorem \ref{thm:generalized-Geroch} can be removed under the additional assumption that $M$ is spin. In this case, one can perform a torical symmetrization in high dimensions and apply the classical result on the Geroch conjecture by Gromov and Lawson. The corresponding rigidity statement can then be derived using a strong maximum principle argument for Ricci flow coupled with a heat equation. The key observation is that by evolving the Riemannian metric \(g\) via the Ricci flow and the potential function \(\varphi\) via the heat equation, the stabilized scalar curvature becomes a supersolution to the heat equation.
\begin{theorem}{\label{thm:generalized-Geroch-spin}}
Let \((M^n, g)\) be a closed, connected, orientable, and spin Riemannian manifold, and let $\varphi$ be a smooth function on $M$. Suppose there exists a smooth map $M\to T^{n}$ of non-zero degree. Then
    \[
    \inf_M (-2\Delta_M\varphi-|\nabla^M\varphi|^2+R_M)\le 0.
    \]
    Moreover, equality holds if and only if the universal cover of $M$ is isometric to $\mathbb{R}^n$ with the flat metric, and $\varphi$ is constant on $M$.
\end{theorem}

The paper is organized as follows. In Section \ref{sec:fb-systol}, we review Brendle and Hung's proof of the systolic inequality stated in Theorem \ref{thm:free-boundary} and establish its rigidity case in Section \ref{sec:fb-rigid}. In Section \ref{sec:systol}, we prove the systolic inequality presented in Theorem \ref{thm:Zhu} and discuss its corresponding rigidity case in Section \ref{sec:systol-rigid}. We complete the proof of Theorem \ref{thm:generalized-Geroch} in Section \ref{sec:geroch} using the weighted minimal slicing argument and present the proof of Theorem \ref{thm:generalized-Geroch-spin} in Section \ref{sec:spin}. Finally, in Appendix \ref{sec:RF}, we derive the evolution equation for the stabilized scalar curvature under the Ricci flow coupled with the heat equation, which may be of independent interest.

\subsection*{Acknowledgement} The author wishes to express sincere gratitude to his advisor, Simon Brendle, for inspiring discussions and continued support.

\section{The Free Boundary Systolic Inequality}{\label{sec:fb-systol}}
In this section, we recall the key elements of Brendle and Hung's proof \cite{Brendle-Hung} of the free boundary systolic inequality, which will also be useful when establishing the corresponding rigidity result. 

Assume $M$ is a compact, connected, orientable manifold of dimension $n$ with nonempty boundary $\partial M$. Let $\eta$ be the outward-pointing unit normal vector along $\partial M$. Let $\varphi$ be a smooth function defined on $M$. Suppose there is a smooth map $\xi: \partial M \to S^1$, and a smooth map $(\theta_1, \ldots, \theta_{n-2}):  M \to T^{n-2}$ such that
\[
(\xi, \theta_1, \ldots, \theta_{n-2}) : \partial M \to S^1 \times T^{n-2}
\]
has nonzero degree. Let $\Xi$ denote the pullback of the volume form on $S^1$ via $\xi$, and for each $i$, let $\Theta_i$ be the pullback of the volume form on $S^1$ via $\theta_i :  M \to S^1$. Thus,
\[
\int_{\partial M} \Xi \wedge \Theta_1 \wedge \cdots \wedge \Theta_{n-2} \neq 0.
\]
We define
\[
\sigma(M,g) := \inf \{\mathcal{H}^1(\alpha) : \alpha \text{ is a smooth, closed curve on } \partial M \text{ with } \int_{\alpha} \Xi \neq 0 \}.
\]

The following slicing argument follows from Section~4 in \cite{Brendle-Hung} (or its proof). We refer the reader to \cite{Federer} and \cite{Gruter} for the regularity theory of free boundary minimal surfaces.

\begin{proposition}\label{prop:fb-slicing}
    Suppose 
    \[
    \inf_{\partial M}\bigl(H_{\partial M} + \langle \nabla^M \varphi, \eta \rangle \bigr) > 0.
    \]
    Then we can find a collection of compact, connected, orientable submanifolds $\Sigma_k$ with boundary and a sequence of positive smooth functions $u_k: \Sigma_k \to \mathbb{R}$ for $k \in \{1, \ldots, n-2\}$, as well as a sequence of positive smooth functions $\rho_k: \Sigma_k \to \mathbb{R}$ for $k \in \{0, 1, \ldots, n-2\}$, satisfying the following properties:
    \begin{enumerate}[(i)]
        \item $\Sigma_0 = M$ and $\rho_0 = e^{\varphi}$.
        \item For each $k \in \{0, \ldots, n-2\}$, $\dim \Sigma_k = n-k$.
        \item For each $k \in \{1, \ldots, n-2\}$, we have $\partial \Sigma_k \subset \partial \Sigma_{k-1}$, and $\Sigma_k$ meets $\partial \Sigma_{k-1}$ orthogonally along $\partial \Sigma_k$.
        \item For each $k \in \{1, \ldots, n-2\}$, the outward-pointing unit normal vector field of $\partial \Sigma_k$ in $\Sigma_{k-1}$ equals $\eta$. Moreover, the second fundamental form of $\partial \Sigma_k$ in $\Sigma_{k-1}$ equals the restriction of $h_{\partial M}$ to $T(\partial \Sigma_k)$.
        \item For each $k \in \{1, \ldots, n-2\}$,
        \[
        \int_{\partial \Sigma_k} \Xi \wedge \Theta_{k+1} \wedge \cdots \wedge \Theta_{n-2} \neq 0.
        \]
        \item For each $k \in \{1, \ldots, n-2\}$, $\Sigma_k$ is a free boundary homologically area minimizer in $(\Sigma_{k-1}, \rho_{k-1}^{\frac{2}{n-k}} g_{\Sigma_{k-1}})$.
        \item For each $k \in \{1, \ldots, n-2\}$,
        \[
        H_{\Sigma_k} + \langle \nabla^{\Sigma_{k-1}}\log\rho_{k-1}, \nu_{\Sigma_k} \rangle = 0.
        \]
        \item For each $k \in \{1, \ldots, n-2\}$, the function $u_k$ satisfies
        \begin{align*}
            &-\Delta_{\Sigma_k} u_k - \ric_{\Sigma_{k-1}}(\nu_{\Sigma_k}, \nu_{\Sigma_k}) u_k - |h_{\Sigma_k}|^2 u_k \\
            &\quad + (D^2_{\Sigma_{k-1}}\log\rho_{k-1})(\nu_{\Sigma_k}, \nu_{\Sigma_k}) u_k - \langle \nabla^{\Sigma_k}\log\rho_{k-1}, \nabla^{\Sigma_k} u_k \rangle = \lambda_k u_k,
        \end{align*}
        with the Neumann boundary condition
        \[
        \langle \nabla^{\Sigma_k} u_k, \eta \rangle - h_{\partial M}(\nu_{\Sigma_k}, \nu_{\Sigma_k}) u_k = 0.
        \]
        Here $\lambda_k \geq 0$ is a nonnegative constant.
        \item For each $k \in \{1, \ldots, n-2\}$, we have $\rho_k = \rho_{k-1}|_{\Sigma_k} \cdot u_k$.
        \item For each $k \in \{0, \ldots, n-2\}$,
        \begin{align*}
            &-2\Delta_{\Sigma_k}\log\rho_k - |\nabla^{\Sigma_k}\log\rho_k|^2 + R_{\Sigma_k} \\
            &\quad + 2\Delta_M\varphi + |\nabla^M\varphi|^2 - R_M \\
            &= \sum_{j=1}^k |\nabla^{\Sigma_j}\log u_j|^2 + \sum_{j=1}^k |h_{\Sigma_j}|^2 + 2\sum_{j=1}^k \lambda_j.
        \end{align*}
        \item For each $k \in \{0, \ldots, n-2\}$,
        \[
        \langle \nabla^{\Sigma_k}\log\rho_k, \eta \rangle + H_{\partial \Sigma_k} = \langle \nabla^M \varphi, \eta \rangle + H_{\partial M}
        \]
        at each point on $\partial \Sigma_k$. Here $H_{\partial \Sigma_k}$ denotes the mean curvature of $\partial \Sigma_k$ in $\Sigma_k$.
    \end{enumerate}
\end{proposition}

\begin{theorem}[Brendle--Hung \cite{Brendle-Hung}]\label{thm:fb-systol}
    Suppose
    \[
    \inf_{\partial M} \bigl(H_{\partial M} + \langle \nabla^M \varphi, \eta \rangle \bigr) > 0.
    \]
    Then there exists a compact, connected, orientable surface $\Sigma$ with $\int_{\partial \Sigma}\Xi \neq 0$ such that
    \begin{align*}
        &\inf_{M}(-2\Delta_M \varphi - |\nabla^M \varphi|^2 + R_M) \cdot |\Sigma| \\
        &\quad + 2\,\inf_{\partial M}(H_{\partial M} + \langle \nabla^M \varphi, \eta \rangle) \cdot |\partial \Sigma| \leq 4\pi.
    \end{align*}
\end{theorem}

\begin{proof}
    Apply Proposition \ref{prop:fb-slicing} with $k = n - 2$. Denote $\Sigma = \Sigma_{n-2}$ and set $\psi = \log \rho_{n-2}$. From this construction, it follows that $\int_{\partial \Sigma} \Xi \neq 0$ and
    \[
    -2\Delta_{\Sigma}\psi - |\nabla^{\Sigma}\psi|^2 + R_{\Sigma} \geq \inf_{M}(-2\Delta_M \varphi - |\nabla^M \varphi|^2 + R_M)
    \]
    at every point of $\Sigma$.

    Since $\Sigma$ is connected, its Euler characteristic is at most 1. Integrating the inequality above and applying the Gauss--Bonnet theorem gives
    \begin{align*}
        4\pi &\geq \int_{\Sigma} R_{\Sigma} \, + 2 \int_{\partial \Sigma} H_{\partial \Sigma}  \\
        &\geq \int_{\Sigma}(-2\Delta_{\Sigma}\psi - |\nabla^{\Sigma}\psi|^2 + R_{\Sigma}) \,  
        + 2\int_{\partial \Sigma}(H_{\partial \Sigma} + \langle \nabla^{\Sigma}\psi, \eta \rangle)  \\
        &\geq \inf_{M}(-2\Delta_M \varphi - |\nabla^M \varphi|^2 + R_M) \cdot |\Sigma| \\
        &\quad + 2\,\inf_{\partial \Sigma}(H_{\partial \Sigma} + \langle \nabla^{\Sigma}\psi, \eta \rangle) \cdot |\partial \Sigma|.
    \end{align*}
    Since 
    \[
    H_{\partial \Sigma} + \langle \nabla^{\Sigma}\psi, \eta \rangle = H_{\partial M} + \langle \nabla^M \varphi, \eta \rangle
    \]
    at each point of $\partial \Sigma$, the theorem follows.
\end{proof}

\begin{corollary}\label{cor:fb-systol}
    Suppose 
    \[
    -2\Delta_M \varphi - |\nabla^M \varphi|^2 + R_M \geq 0
    \]
    at every point in $M$. Then
    \[
    \inf_{\partial M}(H_{\partial M} + \langle \nabla^M \varphi, \eta \rangle) \cdot \sigma(M,g) \leq 2\pi.
    \]
\end{corollary}

\section{The Equality Case of Theorem \ref{thm:free-boundary}}\label{sec:fb-rigid}

In this section, we examine the equality case of Theorem \ref{thm:free-boundary}. Our approach is similar in spirit to the rigidity arguments presented in \cite{Ambrozio}. We restate the result here for convenience.

\begin{theorem}\label{thm:fb-rigidity}
    Suppose 
    \[
    -2\Delta_M \varphi - |\nabla^M \varphi|^2 + R_M \geq 0
    \]
    and
    \[
    \inf_{\partial M} \bigl(H_{\partial M} + \langle \nabla^M \varphi, \eta \rangle \bigr) \cdot \sigma(M, g) = 2\pi.
    \]
    Then $\varphi$ is constant on $M$, and the universal cover of $M$ is isometric to $B^2(r) \times \mathbb{R}^{n-2}$ with the standard product metric for some $r > 0$.
\end{theorem}
We proceed by induction on the dimension \( n \). The case \( n = 2 \) follows directly from the Gauss--Bonnet theorem. Assume Theorem \ref{thm:fb-rigidity} holds for all Riemannian manifolds with dimension at most \( n - 1 \).

Now consider an \( n \)-dimensional manifold \((M, g)\) satisfying
\[
-2\Delta_M \varphi - |\nabla^M \varphi|^2 + R_M \geq 0
\]
and
\[
\inf_{\partial M} \bigl(H_{\partial M} + \langle \nabla^M \varphi, \eta \rangle \bigr) \cdot \sigma(M, g) = 2\pi.
\]
After rescaling the metric, we may assume 
\[
\inf_{\partial M}(H_{\partial M} + \langle \nabla^M \varphi, \eta \rangle) = 1
\quad \text{and} \quad
\sigma(M, g) = 2\pi.
\]

We will use the following identity, which is a consequence of the Gauss equation. For a proof, we refer to Proposition~4.2 in \cite{Brendle-Hung}.

\begin{proposition}\label{prop:Gauss}
    Let $(M,g)$ be a Riemannian manifold, and let $\Sigma \subset M$ be a hypersurface. Suppose $\rho$ is a positive smooth function on $M$, and let $u \in C^{\infty}(\Sigma)$ be a positive function on $\Sigma$. Define $v = \rho|_{\Sigma} \cdot u$. Then, at each point on $\Sigma$, the following identity holds:
    \begin{align*}
        &-2\Delta_{\Sigma} \log v - |\nabla^{\Sigma}\log v|^2 + R_{\Sigma} - (H_{\Sigma} + \langle \nabla^M \log \rho, \nu_{\Sigma} \rangle)^2 \\
        &\quad + 2\Delta_M \log \rho + |\nabla^M \log \rho|^2 - R_M - |\nabla^{\Sigma} \log u|^2 - |h_{\Sigma}|^2 \\
        &= -2u^{-1}\Delta_{\Sigma}u - 2\ric_{M}(\nu_{\Sigma},\nu_{\Sigma})u - 2|h_{\Sigma}|^2 \\
        &\quad + (D^2_{\Sigma}\log \rho)(\nu_{\Sigma},\nu_{\Sigma}) - \langle \nabla^{\Sigma}\log \rho,\nabla^{\Sigma}\log u\rangle.
    \end{align*}
\end{proposition}

\begin{lemma}\label{lem:fb-inductive-rigidity}
    Suppose $\Sigma^{n-1} \subset M$ is a compact, connected, orientable free boundary stable minimal hypersurface in $(M, e^{\frac{2}{n-1}\varphi} g)$ with
    \[
    \int_{\partial \Sigma} \Xi \wedge \Theta_2 \wedge \cdots \wedge \Theta_{n-2} \neq 0.
    \]
    Then:
    \begin{itemize}
        \item The universal cover of $\Sigma$ is isometric to $B^2(1) \times \mathbb{R}^{n-3}$ equipped with the standard product metric.
        \item The function $e^{\varphi}|_{\Sigma} \cdot u$ is constant on $\Sigma$, where $u$ is the first eigenfunction of the Jacobi operator
        \begin{align*}
            &-\Delta_{\Sigma} u - \ric_M(\nu_{\Sigma}, \nu_{\Sigma})u - |h_{\Sigma}|^2 u \\
            &\quad + (D^2_{M}\log \rho)(\nu_{\Sigma}, \nu_{\Sigma})u - \langle \nabla^{\Sigma}\log \rho, \nabla^{\Sigma}u \rangle = \lambda u,
        \end{align*}
        subject to the Neumann boundary condition
        \[
        \langle \nabla^{\Sigma}u, \eta \rangle - h_{\partial M}(\nu_{\Sigma}, \nu_{\Sigma})u = 0.
        \]
    \end{itemize}
\end{lemma}

\begin{proof}
    First, we have the trivial inequality 
    \[
    \sigma(\Sigma, g|_{\Sigma}) \geq \sigma(M,g) = 2\pi.
    \]
    On the other hand, Proposition \ref{prop:Gauss} implies that for $\rho = e^{\varphi}|_{\Sigma} \cdot u$,
    \begin{align*}
        &-2\Delta_{\Sigma}\log \rho - |\nabla^{\Sigma}\log\rho|^2 + R_{\Sigma} \\
        &\ge \inf_{M}(-2\Delta_M \varphi - |\nabla^M\varphi|^2 + R_M) \\
        &= 0
    \end{align*}
    at every point in $\Sigma$. Theorem \ref{thm:fb-systol} then implies $\sigma(\Sigma, g|_{\Sigma}) \leq 2\pi$ and
    \[
    2\pi = \inf_{\partial \Sigma} (H_{\partial \Sigma} + \langle \nabla^{\Sigma}\log \rho, \eta \rangle) \cdot \sigma(\Sigma,g|_{\Sigma}).
    \]
    The claim now follows from the induction hypothesis.
\end{proof}

\begin{lemma}\label{lem:fb-Jacobi-operator}
    Suppose \(\Sigma^{n-1} \subset M\) is a compact, connected, orientable, free boundary stable minimal hypersurface in \((M, e^{\frac{2}{n-1}\varphi} g)\) with
    \[
    \int_{\partial \Sigma} \Xi \wedge \Theta_2 \wedge \cdots \wedge \Theta_{n-2} \neq 0.
    \]
    Then:
    \begin{itemize}
        \item \(\Sigma\) is totally geodesic, with \(\ric_M(\nu_{\Sigma}, \nu_{\Sigma}) = 0\) and \(R_M = 0\) at every point of \(\Sigma\).
        \item \(\varphi|_{\Sigma}\) is constant, with \(\nabla^M \varphi = 0\) and \((D^2_M \varphi)(\nu_{\Sigma}, \nu_{\Sigma}) = 0\) at every point of \(\Sigma\).
        \item \(h_{\partial M}(\nu_{\Sigma}, \nu_{\Sigma}) = 0\) and \(H_{\partial \Sigma} = 1\) at each point on \(\partial \Sigma\).
    \end{itemize}
\end{lemma}

\begin{proof}
    By Lemma \ref{lem:fb-inductive-rigidity}, the function \(\rho = e^{\varphi}|_{\Sigma} \cdot u\) is constant on \(\Sigma\), where \(u\) is the first eigenfunction of the stability operator with Neumann boundary conditions. Hence, Proposition \ref{prop:Gauss} gives
      \begin{align*}
        0 &= (-2\Delta_{M} \varphi - |\nabla^{M} \varphi|^2 + R_{M} - R_{\Sigma}) + |\nabla^{\Sigma} \log u|^2 + |h_{\Sigma}|^2 + 2\lambda.
    \end{align*}
    Since 
    \[
    -2\Delta_{M} \varphi - |\nabla^{M} \varphi|^2 + R_{M} \geq 2 = R_{\Sigma},
    \]
    we must have $\nabla^{\Sigma} \log u = 0$ and $h_{\Sigma} = 0$. This implies that $\Sigma$ is totally geodesic and both $u$ and $\varphi|_{\Sigma}$ are constant on $\Sigma$. In particular, $H_{\Sigma} = 0$, and we conclude that 
    \[
    \langle \nabla^M \varphi, \nu_{\Sigma} \rangle = 0.
    \]
    Therefore, $\nabla^M \varphi = 0$, and
    \begin{align*}
        0 &= -2\Delta_M \varphi - |\nabla^{M} \varphi|^2 + R_{M} - R_{\Sigma} \\
          &= -2(D^2_M \varphi)(\nu_{\Sigma}, \nu_{\Sigma}) + \ric_M(\nu_{\Sigma}, \nu_{\Sigma}).
    \end{align*}
    Moreover, since $\lambda = 0$ and the function $u$ satisfies the equation
    \begin{align*}
        -\Delta_{\Sigma} u - \ric_M(\nu_{\Sigma}, \nu_{\Sigma}) u - |h_{\Sigma}|^2 u + (D^2_M \varphi)(\nu_{\Sigma}, \nu_{\Sigma}) u - \langle \nabla^{\Sigma} \varphi, \nabla^{\Sigma} u \rangle &= \lambda u,
    \end{align*}
    we obtain that
    \[
    -\ric_M(\nu_{\Sigma}, \nu_{\Sigma}) + (D^2_M \varphi)(\nu_{\Sigma}, \nu_{\Sigma}) = 0.
    \]
    From this, we conclude that 
    \[
    \ric_M(\nu_{\Sigma}, \nu_{\Sigma}) = (D^2_M \varphi)(\nu_{\Sigma}, \nu_{\Sigma}) = 0.
    \]

    Finally, the Neumann condition states \(h_{\partial M}(\nu_{\Sigma}, \nu_{\Sigma}) = \langle \nabla^{\Sigma}\log u, \eta\rangle = 0\) and
    \[
    H_{\partial \Sigma}=H_{\partial \Sigma} + \langle\nabla^{\Sigma}\log u, \eta\rangle = H_{\partial M} + \langle\nabla^M\varphi, \eta\rangle = 1
    \]
    at each point on \(\partial \Sigma\). This completes the proof.
\end{proof}

\begin{lemma}\label{lem:fb-higher-foliation}
    Suppose $\Sigma^{n-1} \subset M$ is a compact, connected, orientable free boundary stable minimal hypersurface in $(M, e^{\frac{2}{n-1}\varphi} g)$ with
    \[
    \int_{\partial \Sigma} \Xi \wedge \Theta_2 \wedge \cdots \wedge \Theta_{n-2} \neq 0.
    \]
    Then there exists $\delta > 0$ and a smooth map $w: \Sigma \times (-\delta, \delta) \to \mathbb{R}$ with the following properties:
    \begin{itemize}
        \item For each $x \in \Sigma$, we have $w(x,0) = 0$ and $\frac{\partial}{\partial t} w(x,t)\big|_{t=0} = 1$.
        \item For each $t \in (-\delta, \delta)$, we have $\displaystyle \int_{\Sigma} (w(\cdot, t)-t) \, d\vol_g = 0$.
        \item For each $t \in (-\delta, \delta)$, define
        \[
        \Sigma_t := \{\exp_{x}(w(x,t)\nu_{\Sigma}(x)) : x \in \Sigma\}.
        \]
        Then:
        \begin{enumerate}[(i)]
            \item $\partial \Sigma_t \subset \partial M$ and $\Sigma_t$ meets $\partial M$ orthogonally along $\partial M$.
            \item $H_{\Sigma_t} + \langle \nabla^M \varphi, \nu_{\Sigma_t} \rangle$ is constant on $\Sigma_t$.
        \end{enumerate}
    \end{itemize}
\end{lemma}

\begin{proof}
    Fix some $\alpha \in (0,1)$. Given $\psi \in C^{2,\alpha}(\Sigma)$, define
    \[
    \Sigma_{\psi} := \{\exp_{x}(\psi(x)\nu_{\Sigma}(x)) : x \in \Sigma\}
    \]
    with mean curvature $H_{\Sigma_\psi}$ and unit normal $\nu_{\Sigma_\psi}$. Set $\tilde{H}_{\Sigma_\psi} := H_{\Sigma_\psi} + \langle \nabla^M \varphi, \nu_{\Sigma_\psi}\rangle$.

    Consider the map
    \[
    \mathcal{F}: C^{2,\alpha}(\Sigma) \to \mathring{C}^{\alpha}(\Sigma) \times \mathbb{R} \times C^{1,\alpha}(\partial \Sigma)
    \]
    defined by
    \[
    \psi \mapsto \left(\tilde{H}_{\Sigma_\psi} - \frac{1}{\vol(\Sigma)}\int_{\Sigma}\tilde{H}_{\Sigma_\psi}, \frac{1}{\vol(\Sigma)}\int_{\Sigma}\psi, \langle \nu_{\Sigma_{\psi}}, \eta \rangle \right),
    \]
    where $\mathring{C}^{\alpha}(\Sigma)$ denotes the space of $\alpha$-Hölder continuous functions on $\Sigma$ with zero average.

    By Lemma \ref{lem:Jacobi-operator}, the linearization $(D\mathcal{F})_{\psi=0}: C^{2,\alpha}(\Sigma) \to \mathring{C}^{\alpha}(\Sigma) \times \mathbb{R} \times C^{1,\alpha}(\partial \Sigma)$ is given by
    \[
    f \mapsto \left(-\Delta_{\Sigma}f + \frac{1}{\vol(\Sigma)}\int_{\partial\Sigma}\langle \nabla^{\Sigma}f, \eta \rangle, \frac{1}{\vol(\Sigma)}\int_{\Sigma}f, -\langle \nabla^{\Sigma}f, \eta \rangle \right),
    \]
    which is bijective.

    Hence, by the implicit function theorem, there exists $\delta > 0$ and a smooth map $w : \Sigma \times (-\delta,\delta) \to \mathbb{R}$ that satisfies the desired properties.
\end{proof}

By our assumption 
\[
\inf_{\partial M}\bigl(H_{\partial M} + \langle \nabla^M \varphi,\eta\rangle \bigr) > 0,
\]
Proposition \ref{prop:fb-slicing} implies that we can construct a compact, connected, orientable hypersurface $\Sigma \subset M$ which is a free boundary area minimizer in its relative homology class within $(M, e^{\frac{2}{n-1}\varphi} g)$. Moreover, this hypersurface $\Sigma$ satisfies
\[
\int_{\partial \Sigma} \Xi \wedge \Theta_2 \wedge \cdots \wedge \Theta_{n-2} \neq 0.
\]
Let 
\[
\Sigma_t := \{\exp_x(w(x,t)\nu_{\Sigma}(x)) : x \in \Sigma\}
\]
be the foliation constructed in Lemma \ref{lem:fb-higher-foliation} for $t \in (-\delta,\delta)$. Clearly, $\Sigma_t$ is relative homologous to $\Sigma$ hence 
\[
\int_{\partial \Sigma_t} \Xi \wedge \Theta_2 \wedge \cdots \wedge \Theta_{n-2} \neq 0.
\]

We define the lapse function $f_t: \Sigma_t \to \mathbb{R}$ as follows: at the point $\exp_x(w(x,t)\nu_{\Sigma}(x))$, we set
\begin{equation}{\label{eqn:lapse-function}}
f_t := \left\langle \nu_{\Sigma_t}(x), \frac{\partial}{\partial t}\exp_x(w(x,t)\nu_{\Sigma}(x)) \right\rangle,
\end{equation}
where $\nu_{\Sigma_t}(x)$ is the unit normal vector of $\Sigma_t$ at $\exp_x(w(x,t)\nu_{\Sigma}(x))$.

By choosing $\delta$ smaller if necessary, we may assume $f_t > 0$ on $\Sigma_t$ for all $t \in (-\delta,\delta)$. Since $\mu(t) := H_{\Sigma_t} + \langle \nabla^M \varphi, \nu_{\Sigma_t}\rangle$ depends only on $t$, the lapse function $f_t$ satisfies the Jacobi equation
\begin{equation}\label{eqn:fb-lapse-PDE}
\begin{split}
    &-\Delta_{\Sigma_t} f_t - \ric_{M}(\nu_{\Sigma_t},\nu_{\Sigma_t})f_t - |h_{\Sigma_t}|^2 f_t \\
    &\quad + (D^2_{M}\varphi)(\nu_{\Sigma_t}, \nu_{\Sigma_t}) f_t - \langle \nabla^{\Sigma_t}\varphi, \nabla^{\Sigma_t}f_t \rangle = \mu'(t),
\end{split}
\end{equation}
with the Neumann boundary condition
\begin{equation}\label{eqn:fb-lapse-PDE-Neumann}
\langle \nabla^{\Sigma_t}f_t, \eta \rangle = h_{\partial M}(\nu_{\Sigma_t}, \nu_{\Sigma_t}) f_t.
\end{equation}

\begin{lemma}\label{lem:fb-Ht-sign}
    We have 
    \[
    H_{\Sigma_t} + \langle \nabla^M \varphi, \nu_{\Sigma_t}\rangle \geq 0 \text{ for all } t \in (-\delta,0]
    \]
    and
    \[
    H_{\Sigma_t} + \langle \nabla^M \varphi, \nu_{\Sigma_t}\rangle \leq 0 \text{ for all } t \in [0,\delta).
    \]
\end{lemma}

\begin{proof}
    It is sufficient to show that $\mu'(t) \leq 0$ for all $t \in (-\delta,\delta)$, where $\mu(t) := H_{\Sigma_t} + \langle \nabla^M \varphi, \nu_{\Sigma_t}\rangle$.

    Suppose there exists some $t_0 \in (-\delta,\delta)$ with $\mu'(t_0) > 0$. Choose a constant $\tau > 0$ such that 
    \[
    \frac{\mu'(t_0)}{f_{t_0}} \geq \tau
    \]
    at every point of $\Sigma_{t_0}$. Then, from the Jacobi equation \eqref{eqn:fb-lapse-PDE}, we have
    \begin{align*}
        &-\Delta_{\Sigma_{t_0}} f_{t_0} - \ric_{M}(\nu_{\Sigma_{t_0}}, \nu_{\Sigma_{t_0}})f_{t_0} - |h_{\Sigma_{t_0}}|^2 f_{t_0} \\
        &\quad + (D^2_M \varphi)(\nu_{\Sigma_{t_0}}, \nu_{\Sigma_{t_0}})f_{t_0} - \langle \nabla^{\Sigma_{t_0}}\varphi, \nabla^{\Sigma_{t_0}} f_{t_0} \rangle \geq \tau f_{t_0}.
    \end{align*}
    By Proposition \ref{prop:Gauss}, let $v := e^{\varphi}|_{\Sigma_{t_0}} \cdot f_{t_0}$, we obtain
    \begin{align*}
        &-2\Delta_{\Sigma_{t_0}}\log v - |\nabla^{\Sigma_{t_0}}\log v|^2 + R_{\Sigma_{t_0}} \\
        &\geq -2\Delta_M\varphi - |\nabla^M \varphi|^2 + R_M + 2\tau \\
        &\geq 2\tau
    \end{align*}
    at each point of $\Sigma_{t_0}$.

    Applying Theorem \ref{thm:fb-systol} to $\Sigma_{t_0}$, we conclude
    \[
    2 \inf_{\partial \Sigma_{t_0}} \bigl(H_{\partial \Sigma_{t_0}} + \langle \nabla^{\Sigma_{t_0}}\log v, \eta \rangle\bigr) \cdot \sigma(\Sigma_{t_0}, g|_{\Sigma_{t_0}}) \leq 4\pi - 2\tau |\Sigma|.
    \]
    The Neumann boundary condition \eqref{eqn:fb-lapse-PDE-Neumann} implies that
    \[
    H_{\partial \Sigma_{t_0}} + \langle \nabla^{\Sigma_{t_0}}\log v, \eta \rangle = H_{\partial M} + \langle \nabla^M \varphi, \eta \rangle \geq 1
    \]
    at every point on $\partial \Sigma_{t_0}$.

    It follows that $\sigma(\Sigma_{t_0},g|_{\Sigma_{t_0}})\le 2\pi-\tau |\Sigma|$. This yields a contradiction since we must have $\sigma(\Sigma_{t_0}, g|_{\Sigma_{t_0}}) \geq \sigma(M,g) = 2\pi$. Thus, $\mu'(t) \leq 0$ for all $t \in (-\delta,\delta)$.

    Consequently, $H_{\Sigma_t} + \langle \nabla^M \varphi, \nu_{\Sigma_t}\rangle$ is nonincreasing in $t$. Since it equals 0 at $t=0$, we have $H_{\Sigma_t} + \langle \nabla^M \varphi, \nu_{\Sigma_t}\rangle \geq 0$ for all $t \in (-\delta,0]$ and $H_{\Sigma_t} + \langle \nabla^M \varphi, \nu_{\Sigma_t}\rangle \leq 0$ for all $t \in [0,\delta)$.
\end{proof}

\begin{proposition}\label{prop:fb-weighted-area}
    For each $t \in (-\delta,\delta)$, we have $\displaystyle \int_{\Sigma_t} e^{\varphi} = \int_{\Sigma} e^{\varphi}$.
\end{proposition}

\begin{proof}
    Since $\langle \nu_{\Sigma_t}, \eta \rangle = 0$, applying the first variation formula for weighted areas gives:
    \[
    \frac{d}{dt}\left(\int_{\Sigma_t}e^{\varphi}\right) = \int_{\Sigma_t} (H_{\Sigma_t} + \langle \nabla^M\varphi, \nu_{\Sigma_t} \rangle)e^{\varphi} f_t,
    \]
    where $f_t > 0$ for all $t \in (-\delta,\delta)$. By Lemma \ref{lem:fb-Ht-sign}, we have $H_{\Sigma_t} + \langle \nabla^M \varphi, \nu_{\Sigma_t} \rangle \ge 0$ for $t \leq 0$ and $H_{\Sigma_t} + \langle \nabla^M \varphi, \nu_{\Sigma_t} \rangle \le 0$ for $t \geq 0$. This ensures that
    \[
    \int_{\Sigma_t}e^{\varphi} \leq \int_{\Sigma} e^{\varphi}
    \]
    for $t \in (-\delta,\delta)$. 

    On the other hand, since $\Sigma_t$ is relatively homologous to $\Sigma$ for each $t \in (-\delta,\delta)$, the definition of $\Sigma$ as a free boundary area minimizer implies
    \[
    \int_{\Sigma_t} e^{\varphi} \geq \int_{\Sigma} e^{\varphi}
    \]
    for all $t \in (-\delta,\delta)$, completing the proof.
\end{proof}

\begin{corollary}\label{cor:fb-leaf-rigidity}
    For each $t \in (-\delta,\delta)$, the following statements hold:
    \begin{itemize}
        \item $\Sigma_t$ is totally geodesic, with $\ric_M(\nu_{\Sigma_t}, \nu_{\Sigma_t}) = 0$ and $R_M = 0$ at every point on $\Sigma_t$.
        \item $\varphi|_{\Sigma_t}$ is constant, with $\nabla^M \varphi = 0$ and $(D^2_M\varphi)(\nu_{\Sigma_t}, \nu_{\Sigma_t}) = 0$ at every point on $\Sigma_t$.
        \item $h_{\partial M}(\nu_{\Sigma_t}, \nu_{\Sigma_t}) = 0$ and $H_{\partial \Sigma_t} = 1$ at every point on $\partial \Sigma_t$.
        \item The function $f_t$ is constant on $\Sigma_t$.
    \end{itemize}
\end{corollary}

\begin{proof}
    By Proposition \ref{prop:fb-weighted-area}, for each $t \in (-\delta, \delta)$, the hypersurface $\Sigma_t$ is a free boundary area minimizer in its relative homology class within $(M, e^{\frac{2}{n-1}\varphi} g)$, and it satisfies
    \[
    \int_{\partial \Sigma_t} \Xi \wedge \Theta_2 \wedge \cdots \wedge \Theta_{n-2} \neq 0.
    \] 
    Therefore, the first three statements follow directly from Lemma \ref{lem:fb-Jacobi-operator}.

    With these established, the Jacobi equation \eqref{eqn:fb-lapse-PDE} for the lapse function $f_t$ simplifies to $\Delta_{\Sigma_t} f_t = 0$, together with the Neumann boundary condition \eqref{eqn:fb-lapse-PDE-Neumann} stating $\langle \nabla^{\Sigma_t} f_t, \eta \rangle = 0$. Consequently, $f_t$ must be constant on $\Sigma_t$.
\end{proof}

\begin{corollary}\label{cor:fb-local-isometry}
    If $\Sigma \subset M$ is a compact, connected, orientable hypersurface that is a free boundary area minimizer in its relative homology class within $(M, e^{\frac{2}{n-1}\varphi} g)$ and satisfies
    \[
    \int_{\partial \Sigma} \Xi \wedge \Theta_2 \wedge \cdots \wedge \Theta_{n-2} \neq 0,
    \]
    then there exists a neighborhood $U$ of $\Sigma$ that is isometric to a Riemannian product, and $\varphi$ is constant on $U$.
\end{corollary}

\begin{proof}
    By Corollary \ref{cor:fb-leaf-rigidity}, in a neighborhood around $\Sigma$, we have $\nabla^M \varphi = 0$ and the normal vector field $\nu_{\Sigma_t}$ is parallel (see \cite{Ambrozio} and \cite{BBN}). Consequently, a neighborhood $U$ of $\Sigma$ is isometric to a Riemannian product, and $\varphi$ is constant on $U$.
\end{proof}

Finally, we define $\Phi: \Sigma \times \mathbb{R} \to M$ by $\Phi(x,t) = \exp_x(t\nu_{\Sigma}(x))$.

\begin{theorem}\label{thm:fb-covering}
    The function $\varphi$ is constant on $M$, and the map $\Phi$ is a local isometry with $\Phi(\partial \Sigma \times \mathbb{R}) \subset \partial M$.
\end{theorem}

\begin{proof}
    By Corollary \ref{cor:fb-local-isometry}, $\varphi$ is constant in a neighborhood of $\Sigma$, say $\varphi_0$, and there exists a small $\delta > 0$ such that $\Phi|_{\Sigma \times (-\delta,\delta)}$ is a local isometry with $\Phi(\partial \Sigma \times (-\delta,\delta)) \subset \partial M$.

    Define $\tau$ to be the largest real number for which $\Phi|_{\Sigma \times [0,\tau]}$ is a local isometry, $\Phi(\partial \Sigma \times [0,\tau]) \subset \partial M$, and $\varphi = \varphi_0$ on $\Phi(\Sigma \times [0,\tau])$.

    Consider the compact, connected, orientable hypersurface
    \[
    \hat{\Sigma} := \{\exp_x(\tau \nu_{\Sigma}(x)) : x \in \Sigma\},
    \]
    with $\partial \hat{\Sigma} \subset \partial M$, and $\hat{\Sigma}$ meets $\partial M$ orthogonally. Moreover, $\hat{\Sigma}$ is relatively homologous to $\Sigma$ and
    \[
    \int_{\hat{\Sigma}} e^{\varphi} = e^{\varphi_0} \vol(\hat{\Sigma}) = e^{\varphi_0} \vol(\Sigma) = \int_{\Sigma} e^{\varphi}.
    \]
    Therefore, $\hat{\Sigma}$ is also a free boundary area minimizer in its relative homology class within $(M, e^{\frac{2}{n-1}\varphi}g)$, satisfying
    \[
    \int_{\partial \hat{\Sigma}} \Xi \wedge \Theta_2 \wedge \cdots \wedge \Theta_{n-2} \neq 0.
    \]

    By Corollary \ref{cor:fb-local-isometry}, there is a neighborhood of $\hat{\Sigma}$ on which $\varphi$ is constant and isometric to a Riemannian product. Hence, for some small $\delta > 0$, $\Phi|_{\Sigma \times [0,\tau+\delta)}$ is a local isometry with $\Phi(\partial \Sigma \times [0,\tau+\delta)) \subset \partial M$ and $\varphi = \varphi_0$ on $\Phi(\Sigma \times [0,\tau+\delta))$.

    This contradicts the maximality of $\tau$. Thus, we must have $\Phi|_{\Sigma \times [0,\infty)}$ as a local isometry with $\Phi(\partial \Sigma \times [0,\infty)) \subset \partial M$ and $\varphi = \varphi_0$ on $\Phi(\Sigma \times [0,\infty))$. An analogous argument applies to $\Phi|_{\Sigma \times (-\infty,0]}$, concluding that $\Phi$ is a local isometry on $\Sigma \times \mathbb{R}$, and $\varphi$ is constant on $M$.
\end{proof}

We are now ready to complete the proof of Theorem \ref{thm:fb-rigidity}. By Theorem \ref{thm:fb-covering}, the map $\Phi: \Sigma \times \mathbb{R} \to M$ is a local isometry with $\Phi(\partial \Sigma \times \mathbb{R}) \subset \partial M$. Consequently, $\Phi$ is a covering map. Lemma \ref{lem:fb-inductive-rigidity} showed that the universal cover of $(\Sigma, g|_{\Sigma})$ is isometric to $B^2(1) \times \mathbb{R}^{n-3}$ with the standard product metric. Therefore, the universal cover of $(M, g)$ is isometric to $B^2(1) \times \mathbb{R}^{n-2}$, equipped with the standard product metric.

\section{The Systolic Inequality}{\label{sec:systol}}
In this section, we prove the systolic inequality stated in Theorem \ref{thm:Zhu} that extends Zhu's result in Theorem \cite{Zhu}. 

We assume that \((M^n, g)\) is a closed, connected, orientable Riemannian manifold and that \(\varphi\) is a smooth function on \(M\). Suppose there exists a smooth map
\[
(\omega, \theta_1, \ldots, \theta_{n-2}) : M \to S^2 \times T^{n-2}
\]
with nonzero degree. Denote by \(\Omega\) the pullback of the volume form on \(S^2\) under the map \(\omega : M \to S^2\), and by \(\Theta_i\) the pullback of the volume form on \(S^1\) under the map \(\theta_i : M \to S^1\). Hence,
\[
\int_M \Omega \wedge \Theta_1 \wedge \cdots \wedge \Theta_{n-2} \neq 0.
\]
We define
\[
\mathscr{A}(M,g) := \inf \left\{ \mathcal{H}^2(\Sigma) : \Sigma \text{ is a smooth surface of } M \text{ with } \int_{\Sigma} \Omega \neq 0 \right\}.
\]
The following slicing argument is analogous to Proposition \ref{prop:fb-slicing}. For the proof of the slicing argument, we refer to the works of Schoen and Yau \cite{SY-3manifold, SY-higher} and Gromov and Lawson \cite{GL-2}.

\begin{proposition}\label{prop:closed-slicing}
    We can find a collection of compact, connected, orientable submanifolds \(\Sigma_k\), along with a sequence of positive smooth functions \(u_k : \Sigma_k \to \mathbb{R}\) for \(k \in \{1, \ldots, n-2\}\), and a sequence of positive smooth functions \(\rho_k : \Sigma_k \to \mathbb{R}\) for \(k \in \{0, 1, \ldots, n-2\}\), satisfying the following properties:
    \begin{enumerate}[(i)]
        \item \(\Sigma_0 = M\) and \(\rho_0 = e^{\varphi}\).
        \item \(\dim \Sigma_k = n - k\) for each \(k \in \{0, \ldots, n-2\}\).
        \item For each \(k \in \{1, \ldots, n-2\}\),
        \[
        \int_{\Sigma_k} \Omega \wedge \Theta_{k+1} \wedge \cdots \wedge \Theta_{n-2} \neq 0.
        \]
        \item For each \(k \in \{1, \ldots, n-2\}\), \(\Sigma_k\) is a homologically area-minimizing hypersurface in \((\Sigma_{k-1}, \rho_{k-1}^{\frac{2}{n-k}} g_{\Sigma_{k-1}})\).
        \item For each \(k \in \{1, \ldots, n-2\}\),
        \[
        H_{\Sigma_k} + \left\langle \nabla^{\Sigma_{k-1}} \log \rho_{k-1}, \nu_{\Sigma_k} \right\rangle = 0.
        \]
        \item For each \(k \in \{1, \ldots, n-2\}\), the function \(u_k\) satisfies
        \begin{align*}
            &-\Delta_{\Sigma_k} u_k - \ric_{\Sigma_{k-1}}(\nu_{\Sigma_k}, \nu_{\Sigma_k}) u_k - |h_{\Sigma_k}|^2 u_k \\
            &\quad + (D^2_{\Sigma_{k-1}} \log \rho_{k-1})(\nu_{\Sigma_k}, \nu_{\Sigma_k}) u_k - \left\langle \nabla^{\Sigma_k} \log \rho_{k-1}, \nabla^{\Sigma_k} u_k \right\rangle = \lambda_k u_k,
        \end{align*}
        where \(\lambda_k \geq 0\) is a nonnegative constant.
        \item For each \(k \in \{1, \ldots, n-2\}\), we have
        \(
        \rho_k = \rho_{k-1}|_{\Sigma_k} \cdot u_k.
        \)
        \item For each \(k \in \{0, 1, \ldots, n-2\}\),
        \begin{align*}
            &-2\Delta_{\Sigma_k} \log \rho_k - |\nabla^{\Sigma_k} \log \rho_k|^2 + R_{\Sigma_k} \\
            &\quad + 2\Delta_M \varphi + |\nabla^M \varphi|^2 - R_M \\
            &= \sum_{j=1}^k |\nabla^{\Sigma_j} \log u_j|^2 + \sum_{j=1}^k |h_{\Sigma_j}|^2 + 2 \sum_{j=1}^k \lambda_j.
        \end{align*}
    \end{enumerate}
\end{proposition}

\begin{theorem}\label{thm:stabilized-Zhu}
    There exists a closed, connected, orientable surface $\Sigma$ with
    \(\displaystyle
    \int_{\Sigma} \Omega \neq 0
    \)
    and
    \[
    \inf_{M} \bigl(-2\Delta_M \varphi - |\nabla^M \varphi|^2 + R_M \bigr) \cdot |\Sigma| \leq 8\pi.
    \]
\end{theorem}

\begin{proof}
    We apply Proposition \ref{prop:closed-slicing} with $k = n - 2$. Let $\Sigma = \Sigma_{n-2}$ and $\psi = \log \rho_{n-2}$. Then, it follows that
    \(\displaystyle
    \int_{\Sigma} \Omega \neq 0
    \)
    and
    \[
    -2\Delta_{\Sigma} \psi - |\nabla^{\Sigma} \psi|^2 + R_{\Sigma} \geq \inf_{M} \bigl(-2\Delta_M \varphi - |\nabla^M \varphi|^2 + R_M \bigr)
    \]
    at each point on $\Sigma$.

    Since $\Sigma$ is connected, its Euler characteristic is at most 2. Integrating the above inequality and applying the Gauss--Bonnet Theorem yields
    \begin{align*}
        8\pi & \geq \int_{\Sigma} R_{\Sigma} \, d\mathrm{vol}_{\Sigma} \\
        &\geq \int_{\Sigma} \left( -2\Delta_{\Sigma} \psi - |\nabla^{\Sigma} \psi|^2 + R_{\Sigma} \right) \\
        &\geq \inf_{M} \bigl(-2\Delta_M \varphi - |\nabla^M \varphi|^2 + R_M \bigr) \cdot |\Sigma|.
    \end{align*}
    This completes the proof.
\end{proof}

\begin{corollary}\label{cor:stabilized-sys}
    We have
    \[
    \inf_{M} \bigl(-2\Delta_M \varphi - |\nabla^M \varphi|^2 + R_M \bigr) \cdot \mathscr{A}(M,g) \leq 8\pi.
    \]
\end{corollary}

\section{The Equality Case of Theorem \ref{thm:Zhu}}{\label{sec:systol-rigid}}
In this section, we study the equality case of Theorem \ref{thm:Zhu}. Our approach is similar in spirit to the rigidity arguments presented in \cite{BBN}, \cite{intermediate-rigidity}, and \cite{Zhu}. We restate the result here for convenience.

\begin{theorem}\label{thm:stabilized-Zhu-rigidity}
    Suppose
    \[
    \inf_{M} \bigl(-2\Delta_M \varphi - |\nabla^M \varphi|^2 + R_M \bigr) \cdot \mathscr{A}(M,g) = 8\pi.
    \]
    Then \(\varphi\) is constant on \(M\), and the universal cover of \(M\) is isometric to \(S^2(r) \times \mathbb{R}^{n-2}\) with the standard product metric for some \(r > 0\).
\end{theorem}

The proof is similar to that of Theorem \ref{thm:fb-rigidity}, and we only indicate the necessary changes.

We proceed by induction on the dimension \(n\). The result follows from the Gauss--Bonnet Theorem for \(n = 2\), and we assume that Theorem \ref{thm:stabilized-Zhu} holds for all Riemannian manifolds with dimension at most \(n - 1\).

Now suppose \((M, g)\) is an \(n\)-dimensional Riemannian manifold with
\[
\inf_{M} \bigl(-2\Delta_M \varphi - |\nabla^M \varphi|^2 + R_M \bigr) \cdot \mathscr{A}(M,g) = 8\pi.
\]
After rescaling the metric, we may assume that 
\[
\inf_{M} \bigl(-2\Delta_M \varphi - |\nabla^M \varphi|^2 + R_M \bigr) = 2
 \quad \text{and} \quad
\mathscr{A}(M,g) = 4\pi.
\]

\begin{lemma}\label{lem:inductive-rigidity}
    Suppose $\Sigma^{n-1} \subset M$ is a closed, connected, orientable, stable minimal hypersurface in $(M, e^{\frac{2}{n-1}\varphi} g)$ with
    \[
    \int_{\Sigma} \Omega \wedge \Theta_2 \wedge \cdots \wedge \Theta_{n-2} \neq 0.
    \]
    Then:
    \begin{itemize}
        \item The universal cover of $\Sigma$ is isometric to $S^2(1) \times \mathbb{R}^{n-3}$ equipped with the standard product metric.
        \item The function $e^{\varphi}|_{\Sigma} \cdot u$ is constant on $\Sigma$, where $u$ is the first eigenfunction of the Jacobi operator
        \begin{align*}
            &-\Delta_{\Sigma} u - \ric_{M}(\nu_{\Sigma}, \nu_{\Sigma}) u - |h_{\Sigma}|^2 u \\
            &\quad + (D^2_{M} \log \rho)(\nu_{\Sigma}, \nu_{\Sigma}) u - \langle \nabla^{\Sigma} \log \rho, \nabla^{\Sigma} u \rangle = \lambda u.
        \end{align*}
    \end{itemize}
\end{lemma}

\begin{proof}
    First, we have the trivial inequality 
    \[
    \mathscr{A}(\Sigma, g_{\Sigma}) \geq \mathscr{A}(M, g) = 4\pi.
    \]
    On the other hand, Proposition \ref{prop:Gauss} implies that for $\rho = e^{\varphi}|_{\Sigma} \cdot u$,
    \begin{align*}
        &-2\Delta_{\Sigma} \log \rho - |\nabla^{\Sigma} \log \rho|^2 + R_{\Sigma} \\
        &\geq \inf_{M} \bigl(-2\Delta_M \varphi - |\nabla^M \varphi|^2 + R_M \bigr) \\
        &= 2
    \end{align*}
    at each point in $\Sigma$. Theorem \ref{thm:stabilized-Zhu} then implies that
    \(
    \mathscr{A}(\Sigma, g|_{\Sigma}) \leq 4\pi
    \)
    and
    \[
    8\pi = \inf_{M} \bigl(-2\Delta_M \varphi - |\nabla^M \varphi|^2 + R_M \bigr) \cdot \mathscr{A}(M, g).
    \]
    Therefore, we have
    \[
    \inf_{M} \bigl(-2\Delta_M \varphi - |\nabla^M \varphi|^2 + R_M \bigr) \cdot \mathscr{A}(M, g) = 8\pi.
    \]
    The claim now follows from the induction hypothesis.
\end{proof}

\begin{lemma}\label{lem:Jacobi-operator}
    Suppose $\Sigma^{n-1} \subset M$ is a closed, connected, orientable, stable minimal hypersurface in $(M, e^{\frac{2}{n-1}\varphi} g)$ with
    \[
    \int_{\Sigma} \Omega \wedge \Theta_2 \wedge \cdots \wedge \Theta_{n-2} \neq 0.
    \]
    Then:
    \begin{itemize}
        \item $\Sigma$ is totally geodesic, with $\ric_M(\nu_{\Sigma}, \nu_{\Sigma}) = 0$ and $R_M = 0$ at each point in $\Sigma$.
        \item $\varphi|_{\Sigma}$ is constant, with $\nabla^M \varphi = 0$ and $(D^2_M \varphi)(\nu_{\Sigma}, \nu_{\Sigma}) = 0$ at each point in $\Sigma$.
    \end{itemize}
\end{lemma}

\begin{proof}
    The proof is analogous to the proof of Lemma \ref{lem:fb-Jacobi-operator}.    
\end{proof}

\begin{lemma}{\label{lem:higher-foliation}}
    Suppose $\Sigma^{n-1}\subset M$ is a closed, connected, orientable stable minimal hypersurface in $(M,e^{\frac{2}{n-1}\varphi}g)$ with
    \[
\int_{\Sigma}\Omega\wedge \Theta_2\wedge \cdots \wedge \Theta_{n-2}\ne 0.
\]
Then there exists $\delta > 0$ and a smooth map $w: \Sigma \times (-\delta, \delta) \to \mathbb{R}$ with the following properties:
    \begin{itemize}
        \item For each $x \in \Sigma$, we have $w(x,0) = 0$ and $\frac{\partial}{\partial t} w(x,t)\big|_{t=0} = 1$.
        \item For each $t \in (-\delta, \delta)$, we have $\displaystyle \int_{\Sigma} (w(\cdot, t)-t) \, d\vol_g = 0$.
        \item For each $t \in (-\delta, \delta)$, define
        \[
        \Sigma_t := \{\exp_{x}(w(x,t)\nu_{\Sigma}(x)) : x \in \Sigma\}.
        \]
        Then $H_{\Sigma_t} + \langle \nabla^M \varphi, \nu_{\Sigma_t} \rangle$ is constant on $\Sigma_t$.
    \end{itemize}
\end{lemma}
\begin{proof}
    Fix some $\alpha \in (0,1)$. Given $\psi \in C^{2,\alpha}(\Sigma)$, define
    \[
    \Sigma_{\psi} := \{\exp_{x}(\psi(x)\nu_{\Sigma}(x)) : x \in \Sigma\}
    \]
    with mean curvature $H_{\Sigma_\psi}$ and unit normal $\nu_{\Sigma_\psi}$. Set $\tilde{H}_{\Sigma_\psi} := H_{\Sigma_\psi} + \langle \nabla^M \varphi, \nu_{\Sigma_\psi}\rangle$.

    Consider the map
    \[
    \mathcal{F}: C^{2,\alpha}(\Sigma) \to \mathring{C}^{\alpha}(\Sigma) \times \mathbb{R} 
    \]
    defined by
    \[
    \psi \mapsto \left(\tilde{H}_{\Sigma_\psi} - \frac{1}{\vol(\Sigma)}\int_{\Sigma}\tilde{H}_{\Sigma_\psi}, \frac{1}{\vol(\Sigma)}\int_{\Sigma}\psi \right),
    \]
    where $\mathring{C}^{\alpha}(\Sigma)$ denotes the space of $\alpha$-Hölder continuous functions on $\Sigma$ with zero average.

    By Lemma \ref{lem:Jacobi-operator}, the linearization $(D\mathcal{F})_{\psi=0}: C^{2,\alpha}(\Sigma) \to \mathring{C}^{\alpha}(\Sigma) \times \mathbb{R} $ is given by
    \[
    f \mapsto \left(-\Delta_{\Sigma}f , \frac{1}{\vol(\Sigma)}\int_{\Sigma}f\right),
    \]
    which is bijective.

    Hence, by the implicit function theorem, there exists $\delta > 0$ and a smooth map $w : \Sigma \times (-\delta,\delta) \to \mathbb{R}$ that satisfies the desired properties.
\end{proof}
By Proposition \ref{prop:closed-slicing}, we can construct a closed, connected, orientable hypersurface $\Sigma \subset M$ that is area-minimizing in its homology class within $(M, e^{\frac{2}{n-1}\varphi} g)$. Moreover, this hypersurface $\Sigma$ satisfies
\[
\int_{\Sigma} \Omega \wedge \Theta_2 \wedge \cdots \wedge \Theta_{n-2} \neq 0.
\]
Let 
\[
\Sigma_t := \{\exp_x(w(x,t)\nu_{\Sigma}(x)) : x \in \Sigma\}
\]
be the foliation constructed in Lemma \ref{lem:higher-foliation} for $t \in (-\delta,\delta)$. Clearly, $\Sigma_t$ is homologous to $\Sigma$, and hence
\[
\int_{\Sigma_t} \Omega \wedge \Theta_2 \wedge \cdots \wedge \Theta_{n-2} \neq 0.
\]

We define the lapse function $f_t: \Sigma_t \to \mathbb{R}$ analogously to \eqref{eqn:lapse-function}. By choosing $\delta$ smaller if necessary, we may assume $f_t > 0$ on $\Sigma_t$ for all $t \in (-\delta,\delta)$. Since $\mu(t) := H_{\Sigma_t} + \langle \nabla^M \varphi, \nu_{\Sigma_t} \rangle$ depends only on $t$, the lapse function $f_t$ satisfies a Jacobi equation analogous to \eqref{eqn:fb-lapse-PDE}.
\begin{lemma}{\label{lem:Ht-sign}}
     We have 
    \[
    H_{\Sigma_t} + \langle \nabla^M \varphi, \nu_{\Sigma_t}\rangle \geq 0 \text{ for all } t \in (-\delta,0]
    \]
    and
    \[
    H_{\Sigma_t} + \langle \nabla^M \varphi, \nu_{\Sigma_t}\rangle \leq 0 \text{ for all } t \in [0,\delta).
    \]
\end{lemma}

\begin{proof}
    It is sufficient to show that $\mu'(t) \leq 0$ for all $t \in (-\delta,\delta)$, where $\mu(t) := H_{\Sigma_t} + \langle \nabla^M \varphi, \nu_{\Sigma_t}\rangle$.

    Suppose there exists some $t_0 \in (-\delta,\delta)$ with $\mu'(t_0) > 0$. Choose a constant $\tau > 0$ such that 
    \[
    \frac{\mu'(t_0)}{f_{t_0}} \geq \tau
    \]
    at every point of $\Sigma_{t_0}$. Then, from the Jacobi equation \eqref{eqn:fb-lapse-PDE}, we have
    \begin{align*}
        &-\Delta_{\Sigma_{t_0}} f_{t_0} - \ric_{M}(\nu_{\Sigma_{t_0}}, \nu_{\Sigma_{t_0}})f_{t_0} - |h_{\Sigma_{t_0}}|^2 f_{t_0} \\
        &\quad + (D^2_M \varphi)(\nu_{\Sigma_{t_0}}, \nu_{\Sigma_{t_0}})f_{t_0} - \langle \nabla^{\Sigma_{t_0}}\varphi, \nabla^{\Sigma_{t_0}} f_{t_0} \rangle \geq \tau f_{t_0}.
    \end{align*}
    By Proposition \ref{prop:Gauss}, let $v := e^{\varphi}|_{\Sigma_{t_0}} \cdot f_{t_0}$, we obtain
    \begin{align*}
        &-2\Delta_{\Sigma_{t_0}}\log v - |\nabla^{\Sigma_{t_0}}\log v|^2 + R_{\Sigma_{t_0}} \\
        &\geq -2\Delta_M\varphi - |\nabla^M \varphi|^2 + R_M + 2\tau \\
        &\geq 2+ 2\tau
    \end{align*}
    at each point of $\Sigma_{t_0}$.

    Applying Theorem \ref{thm:stabilized-Zhu} to $\Sigma_{t_0}$, we conclude
\[
\mathscr{A}(M,g)\le \mathscr{A}({\Sigma_{t_0}},g|_{{\Sigma_{t_0}}})\le \frac{4\pi}{1+\tau}
\]
This contradicts the assumption that $\mathscr{A}(M,g)=4\pi$.
\end{proof}

\begin{proposition}{\label{prop:weighted-area}}
    For each $t\in (-\delta,\delta)$, we have $\displaystyle \int_{\Sigma_t}e^{\varphi}=\displaystyle \int_{\Sigma}e^{\varphi}$.
\end{proposition}
\begin{proof}
    The proof is analogous to Proposition \ref{prop:fb-weighted-area}.
\end{proof}
\begin{corollary}{\label{cor:leaf-rigidity}}
    For each $t \in (-\delta,\delta)$, the following statements hold:
    \begin{itemize}
        \item $\Sigma_t$ is totally geodesic, with $\ric_M(\nu_{\Sigma_t}, \nu_{\Sigma_t}) = 0$ and $R_M = 2$ at every point on $\Sigma_t$.
        \item $\varphi|_{\Sigma_t}$ is constant, with $\nabla^M \varphi = 0$ and $(D^2_M\varphi)(\nu_{\Sigma_t}, \nu_{\Sigma_t}) = 0$ at every point on $\Sigma_t$.
    \end{itemize}
\end{corollary}
\begin{proof}
    The proof is analogous to the proof of Corollary \ref{cor:fb-leaf-rigidity}.
\end{proof}

\begin{corollary}{\label{cor:local-isometry}}
    If $\Sigma\subset M$ is a closed, connected, orientable hypersurface that is homologically area minimizing in $(M,e^{\frac{2}{n-1}\varphi}g)$ that satisfies
\[
\int_{\Sigma}\Omega\wedge \Theta_2\wedge \cdots \wedge \Theta_{n-2}\ne 0,
\] 
then there exists a neighborhood $U$ of $\Sigma$ that is isometric to a Riemannian product, and $\varphi$ is constant on $U$.
\end{corollary}
\begin{proof}
    The proof is analogous to the proof of Corollary \ref{cor:fb-local-isometry}.
\end{proof}
Finally, we define $\Phi:\Sigma\times \bR\to M$ by $\Phi(x,t)=\exp_x(t\nu_{\Sigma}(x))$.
\begin{theorem}{\label{thm:covering}}
    $\varphi$ is constant on $M$ and the map $\Phi$ is a local isometry.
\end{theorem}
The proof of Theorem \ref{thm:covering} is analogous to the proof of Theorem \ref{thm:fb-covering}. From here on, the proof of Theorem \ref{thm:stabilized-Zhu-rigidity} proceeds in the same way as in the proof of Theorem \ref{thm:fb-rigidity}.

\section{Proof of Theorem \ref{thm:generalized-Geroch}}\label{sec:geroch}
In this section, we examine the rigidity of the generalized Geroch conjecture within the framework of stabilized curvature. Chu and Zhu \cite{Chu-Zhu} established a generalized version of the inequality in Theorem \ref{thm:generalized-Geroch} for all Schoen-Yau-Schick manifolds of dimension at most 7. For simplicity, we will reprove this result in the current setting. The proof of the rigidity statement in Theorem \ref{thm:generalized-Geroch} follows a similar approach to that of Theorem \ref{thm:Zhu}, and we will only highlight the necessary modifications.

We assume that \((M^n, g)\) is a closed, connected, orientable Riemannian manifold and that \(\varphi\) is a smooth function on \(M\). Suppose there exists a map
\[
    (\theta_1, \ldots, \theta_n) : M \to T^{n}
\]
with nonzero degree. Denote by \(\Theta_i\) the pullback of the volume form on \(S^1\) under the map \(\theta_i : M \to S^1\). Hence,
\[
    \int_M \Theta_1 \wedge \cdots \wedge \Theta_n \neq 0.
\]
The following proposition is analogous to Proposition \ref{prop:closed-slicing}. See also the version in Chu and Zhu \cite{Chu-Zhu}.

\begin{proposition}\label{prop:closed-slicing-torus}
    We can find a collection of compact, connected, orientable submanifolds $\Sigma_k$ and a sequence of positive smooth functions $u_k : \Sigma_k \to \mathbb{R}$ for each $k \in \{1, \ldots, n-1\}$, as well as a sequence of positive smooth functions $\rho_k : \Sigma_k \to \mathbb{R}$ for each $k \in \{0, 1, \ldots, n-1\}$, satisfying the following properties:
    \begin{enumerate}[(i)]
        \item $\Sigma_0 = M$ and $\rho_0 = e^{\varphi}$.
        \item $\dim \Sigma_k = n - k$ for each $k \in \{0, \ldots, n-1\}$.
        \item For each $k \in \{1, \ldots, n-1\}$,
        \[
        \int_{\Sigma_k} \Theta_{k+1} \wedge \cdots \wedge \Theta_{n} \neq 0.
        \]
        \item For each $k \in \{1, \ldots, n-1\}$, the hypersurface $\Sigma_k$ is homologically area-minimizing in $(\Sigma_{k-1}, \rho_{k-1}^{\frac{2}{n-k}} g_{\Sigma_{k-1}})$.
        \item For each $k \in \{1, \ldots, n-1\}$,
        \[
        H_{\Sigma_k} + \left\langle \nabla^{\Sigma_{k-1}} \log \rho_{k-1}, \nu_{\Sigma_k} \right\rangle = 0.
        \]
        \item For each $k \in \{1, \ldots, n-1\}$, the function $u_k$ satisfies
        \begin{align*}
            &-\Delta_{\Sigma_k} u_k - \ric_{\Sigma_{k-1}}(\nu_{\Sigma_k}, \nu_{\Sigma_k}) u_k - |h_{\Sigma_k}|^2 u_k \\
            &\quad + (D^2_{\Sigma_{k-1}} \log \rho_{k-1})(\nu_{\Sigma_k}, \nu_{\Sigma_k}) u_k - \left\langle \nabla^{\Sigma_k} \log \rho_{k-1}, \nabla^{\Sigma_k} u_k \right\rangle = \lambda_k u_k,
        \end{align*}
        where $\lambda_k \geq 0$ is a nonnegative constant.
        \item For each $k \in \{1, \ldots, n-1\}$, we have $
        \rho_k = \rho_{k-1}|_{\Sigma_k} \cdot u_k$.
        
        \item For each $k \in \{0, 1, \ldots, n-1\}$, we have
        \begin{align*}
            &-2\Delta_{\Sigma_k} \log \rho_k - |\nabla^{\Sigma_k} \log \rho_k|^2 + R_{\Sigma_k} \\
            &\quad + 2\Delta_M \varphi + |\nabla^M \varphi|^2 - R_M \\
            &= \sum_{j=1}^k |\nabla^{\Sigma_j} \log u_j|^2 + \sum_{j=1}^k |h_{\Sigma_j}|^2  + 2 \sum_{j=1}^k \lambda_j.
        \end{align*}
    \end{enumerate}
\end{proposition}

We apply Proposition \ref{prop:closed-slicing-torus} with \(k = n-1\). Let \(\Sigma = \Sigma_{n-1}\) and \(\psi = \log \rho_{n-1}\). It follows that
\[
-2\Delta_{\Sigma} \psi - |\nabla^{\Sigma} \psi|^2 + R_{\Sigma} \geq \inf_{M} \left( -2\Delta_M \varphi - |\nabla^M \varphi|^2 + R_M \right)
\]
holds at each point on \(\Sigma\). Since \(\dim \Sigma = 1\), we have \(R_{\Sigma} = 0\). Integrating the above inequality over \(\Sigma\) yields the inequality stated in Theorem \ref{thm:generalized-Geroch}.

Next, we assume that Theorem \ref{thm:generalized-Geroch} holds for all Riemannian manifolds with dimension at most \(n-1\). Let \(M\) be an \(n\)-dimensional Riemannian manifold satisfying
\[
\inf_{M} \left( -2\Delta_M \varphi - |\nabla^M \varphi|^2 + R_M \right) = 0.
\]
\begin{lemma}{\label{lem:inductive-rigidity-torus}}
    Suppose $\Sigma^{n-1}\subset M$ is a closed, connected, orientable stable minimal hypersurface in $(M,e^{\frac{2}{n-1}\varphi}g)$ with
    \[
\int_{\Sigma} \Theta_2\wedge \cdots \wedge \Theta_{n}\ne 0.
\]
Then,
    \begin{itemize}
        \item The universal cover of $\Sigma$ is isometric to $ \bR^{n-1}$ equipped with flat metric.
        \item The function $e^{\varphi}|_{\Sigma}\cdot u$ is constant on $\Sigma$, where $u$ is the first eigenfunction of the Jacobi operator
        \begin{align*}
        &-\lp_{\Sigma}u-\ric_{M}(\nu_{\Sigma},\nu_{\Sigma})u-|h_{\Sigma}|^2u\\
        &\quad +(D^2_{M}\log \rho)(\nu_{\Sigma},\nu_{\Sigma})u-\la \gd^{\Sigma}\log \rho,\gd^{\Sigma}u\rg =\lambda u.
    \end{align*}
        \item $\Sigma$ is totally geodesic, with $\ric_M(\nu_{\Sigma},\nu_{\Sigma})=0$ and $R_M=0$ at each point in $\Sigma$. 
        \item $\varphi|_{\Sigma}$ is constant, with $\gd^M\varphi=0$ and $D^2_M\varphi(\nu_{\Sigma},\nu_{\Sigma})=0$ at each point in $\Sigma$.
    \end{itemize}
\end{lemma}
\begin{proof}
    Proposition \ref{prop:Gauss} implies that for $\rho=e^{\varphi}|_{\Sigma}\cdot u$,
    \begin{align*}
    &-2\lp_{\Sigma}\log \rho-|\gd^{\Sigma}\log\rho|^2+R_{\Sigma}\\
    &\ge \inf_M(-2\lp_M\varphi-|\gd^M\varphi|^2+R_M)\\
    &=0
    \end{align*}
    at each point in $\Sigma$. Applying the induction hypothesis and a similar argument as in Lemma \ref{lem:Jacobi-operator} yields the desired conclusions.
\end{proof}
By Proposition \ref{prop:closed-slicing-torus}, we can construct a closed, connected, orientable hypersurface $\Sigma\subset M$ that is area-minimizing in its homology class within $(M,e^{\frac{2}{n-1}\varphi}g)$. Moreover, this hypersurface $\Sigma$ satisfies
\[
\int_{\Sigma} \Theta_2\wedge \cdots \wedge \Theta_{n}\ne 0.
\] 
Let 
\[
        \Sigma_t=\{\exp_{x}(w(x,t)\nu_{\Sigma}(x)):x\in \Sigma\}
\]
be the foliation constructed analogously to Lemma \ref{lem:higher-foliation} for $t\in (-\delta,\delta)$, such that $\mu(t):=H_{\Sigma_t}+\la \gd^M\varphi,\nu_{\Sigma_t}\rg$ depends only on $t$. Then ,$\Sigma_t$ is homologous to $\Sigma$, and
\[
\int_{\Sigma_t} \Theta_2\wedge \cdots \wedge \Theta_{n}\ne 0.
\]

We define the lapse function $f_t:\Sigma_t \to \bR$ analogously to \eqref{eqn:lapse-function}. By choosing $\delta$ smaller if necessary, we may assume $f_t>0$ on $\Sigma_t$ for all $t\in (-\delta,\delta)$. The lapse function $f_t$ satisfies a Jacobi equation analogous to \eqref{eqn:fb-lapse-PDE}.
\begin{lemma}{\label{lem:Ht-sign-torus}}
We have 
    \[
    H_{\Sigma_t} + \langle \nabla^M \varphi, \nu_{\Sigma_t}\rangle \geq 0 \text{ for all } t \in (-\delta,0]
    \]
    and
    \[
    H_{\Sigma_t} + \langle \nabla^M \varphi, \nu_{\Sigma_t}\rangle \leq 0 \text{ for all } t \in [0,\delta).
    \]
\end{lemma}

\begin{proof}
    It is sufficient to show that $\mu'(t) \leq 0$ for all $t \in (-\delta,\delta)$, where $\mu(t) := H_{\Sigma_t} + \langle \nabla^M \varphi, \nu_{\Sigma_t}\rangle$.

    Suppose there exists some $t_0 \in (-\delta,\delta)$ with $\mu'(t_0) > 0$. Choose a constant $\tau > 0$ such that 
    \[
    \frac{\mu'(t_0)}{f_{t_0}} \geq \tau
    \]
    at every point of $\Sigma_{t_0}$. Then, a similar argument as in Lemma \ref{lem:Ht-sign} shows that the function \(v := e^{\varphi}|_{\Sigma_{t_0}} \cdot f_{t_0}\) satisfies
\begin{align*}
    &-2\Delta_{\Sigma_{t_0}} \log v - |\nabla^{\Sigma_{t_0}} \log v|^2 + R_{\Sigma_{t_0}} \\
    &\geq -2\Delta_M \varphi - |\nabla^M \varphi|^2 + R_M + 2\tau \\
    &\geq 2\tau
\end{align*}
at each point on \(\Sigma_{t_0}\). 
This leads to a contradiction.
\end{proof}

\begin{corollary}\label{cor:local-isometry-torus}
    If \(\Sigma \subset M\) is a closed, connected, orientable hypersurface that is homologically area-minimizing in \((M, e^{\frac{2}{n-1}\varphi} g)\), and satisfies
    \[
    \int_{\Sigma} \Theta_2 \wedge \cdots \wedge \Theta_{n} \neq 0,
    \]
    then there exists a neighborhood \(U\) of \(\Sigma\) that is isometric to a Riemannian product, and \(\varphi\) is constant on \(U\).
\end{corollary}

The proof of Corollary \ref{cor:local-isometry-torus} is analogous to that of Corollary \ref{cor:local-isometry}. Finally, we define the map \(\Phi : \Sigma \times \mathbb{R} \to M\) by
\[
\Phi(x, t) = \exp_x(t \nu_{\Sigma}(x)).
\]
Analogous to Theorem \ref{thm:covering}, we can show that \(\varphi\) is constant on \(M\) and that the map \(\Phi\) is a local isometry. From here, the proof of Theorem \ref{thm:generalized-Geroch} proceeds in the same manner as the proof of Theorem \ref{thm:stabilized-Zhu-rigidity}.

\section{Proof of Theorem \ref{thm:generalized-Geroch-spin}}{\label{sec:spin}}
In this section, we assume \(M^n\) is closed, connected, spin, and admits a degree non-zero map to the \(n\)-dimensional torus \(T^n\). We first prove the inequality of Theorem \ref{thm:generalized-Geroch-spin}.  

\begin{proposition}\label{prop:strict-inequality}
    Let \(g\) be a Riemannian metric on \(M\) and let \(\varphi\) be a smooth function on \(M\). Then,
    \[
        \inf_{M} \left( -2\lp_M \varphi - |\gd^M \varphi|^2 + R_M \right) \leq 0.
    \]
\end{proposition}

\begin{proof}
    Suppose there exists a Riemannian metric \(g\) on \(M\) and a smooth function \(\varphi \in C^{\infty}(M)\) such that
    \[
        -2\lp_M \varphi - |\gd^M \varphi|^2 + R_M > 0.
    \]
    Choose a sufficiently large positive integer \(N\) such that
    \[
        -2\lp_M \varphi - \frac{N+1}{N} |\gd^M \varphi|^2 + R_M > 0.
    \]
    This implies that the Riemannian manifold \((M \times T^{N}, g + e^{\frac{2\varphi}{N}} g_{T^N})\) has strictly positive scalar curvature, where \(g_{T^N}\) is a flat metric on the \(N\)-dimensional torus. However, \(M \times T^N\) is a spin manifold that admits a degree non-zero map to \(T^{n+N}\). This leads to a contradiction by the work of Gromov and Lawson \cite{GL-2}.
\end{proof}

We are now ready to prove Theorem \ref{thm:generalized-Geroch-spin}. We evolve \(g\) by the Ricci flow and \(\varphi\) by the heat equation
\[
    \begin{cases}
        \frac{\partial }{\partial t}g(t) = -2\ric_{g(t)}, & g(0) = g, \\
        \frac{\partial }{\partial t}\varphi(t) = \lp_{g(t)} \varphi(t), & \varphi(0) = \varphi.
    \end{cases}
\]
Corollary \ref{cor:rf-monotonicity} implies that the quantity \(S := -2\lp_{g(t)} \varphi(t) - |d \varphi|_{g(t)}^2 + R_{g(t)}\) satisfies the equation
\[
    \frac{\partial}{\partial t} S - \lp S = 2|\ric - D^2 \varphi|^2.
\]
By the strong maximum principle and Proposition \ref{prop:strict-inequality}, we must have \(\ric_g = D^2 \varphi\) and \(-2\lp_g \varphi - |\gd^M \varphi|_g^2 + R_g = 0\) at each point of \(M\). In particular, we must have \(-\lp_g \varphi - |\gd^M \varphi|_g^2 = 0\) identically. Integrating over \(M\) implies that \(\varphi\) is constant. Hence, the metric \(g\) must be flat.

\appendix
\section{Evolution of Stabilized Scalar Curvature under Ricci Flow Coupled with the Heat Equation}\label{sec:RF}
In this section, we show that the infimum of the stabilized scalar curvature is monotone under the Ricci flow coupled with the heat equation. Let \((M^n, g)\) be a closed Riemannian manifold evolved by the Ricci flow
\[
    \frac{\partial g(t)}{\partial t} = -2\ric_{g(t)},
\]
and let \(\varphi(x,t): M \to \mathbb{R}\) be a time-dependent smooth function. We define
\[
    S = 2\Delta_{g(t)} \varphi(t) - |\nabla \varphi|_{g(t)}^2 + R_{g(t)}
\]
to be the stabilized scalar curvature at time \(t\). Note that we are using a slightly different sign convention here compared to other parts of the article, which is essentially equivalent by replacing \(\varphi\) with \(-\varphi\) in the definition. We refer the reader to \cite{perelman2002entropyformularicciflow} and \cite{coupled-flow} for general results on Ricci flows coupled with other parabolic equations.
\begin{lemma}
    We have
    \begin{align*}{\label{eqn:monotone-stablized-scalar}}
        &\frac{\p S}{\p t}-\lp S\\
        &=2|\ric+D^2\varphi|^2+2\lp \left(\frac{\p}{\p t}\varphi-\lp\varphi\right)-2\p^i \varphi\p_i \left(\frac{\p}{\p t}\varphi-\lp\varphi\right).
    \end{align*}
\end{lemma}
\begin{proof}
 It is also well-known that for any function $\varphi$, the following holds:
    \[
    \frac{\p}{\p t}\lp\varphi-\lp\frac{\p}{\p t}\varphi=2\ric^{ij}D^2_{i,j}\varphi.
    \]
    This implies
    \[
   \frac{\p}{\p t}\lp\varphi-\lp\lp\varphi =\lp \left(\frac{\p}{\p t}\varphi-\lp\varphi\right)+2\ric^{ij}D^2_{i,j}\varphi.
    \]
    Additionally, the Bochner formula gives
    \begin{align*}
        &\frac{\p}{\p t}|d\varphi|^2-\lp|d\varphi|^2\\
        &=2\p^i \varphi\p_i\frac{\p}{\p t}\varphi+2\ric_{ij}\p^i\varphi \p^j \varphi\\
        &-2\p^i \varphi\p_i\lp \varphi-2|D^2\varphi|^2-2\ric_{ij}\p^i \varphi\p^j\varphi \\
        &=2\p^i \varphi\p_i \left(\frac{\p}{\p t}\varphi-\lp\varphi\right)-2|D^2\varphi|^2.
    \end{align*}
    Finally, recall that the evolution of the scalar curvature along the Ricci flow is given by
    \[
   \frac{\p}{\p t}R-\lp R=2|\ric|^2.
    \]
    Combining these three identities, we obtain
    \begin{align*}
        &\frac{\p}{\p t}-\lp S\\
        &=4\ric^{ij}D^2_{i,j}\varphi+2|D^2\varphi|^2+2|\ric|^2\\
        &+2\lp \left(\frac{\p}{\p t}\varphi-\lp\varphi\right)-2\p^i \varphi\p_i \left(\frac{\p}{\p t}\varphi-\lp\varphi\right)\\
        &=2|\ric+D^2\varphi|^2+2\lp \left(\frac{\p}{\p t}\varphi-\lp\varphi\right)-2\p^i \varphi\p_i \left(\frac{\p}{\p t}\varphi-\lp\varphi\right).
    \end{align*}
    This completes the proof.
\end{proof}
\begin{corollary}\label{cor:rf-monotonicity}
    If we evolve \(\varphi\) by the heat equation
    \[
            \frac{\partial }{\partial t}\varphi(t) = \Delta_{g(t)} \varphi(t),
    \]
    then 
    \[
        \frac{\partial S}{\partial t} - \Delta S = 2|\ric + D^2 \varphi|^2.
    \]
\end{corollary}
\begin{remark}\label{rmk:Perelman}
    Corollary \ref{cor:rf-monotonicity} can be compared to Perelman's calculation of the \(\cF\)-functional in \cite{perelman2002entropyformularicciflow}. Instead of solving the forward equation, Perelman considers \(\varphi\) as the solution to the backward heat equation
    \[
        \frac{\partial \varphi}{\partial t} + \lp \varphi - |\gd \varphi|^2 + R = 0,
    \]
    and shows that the quantity \(S e^{-\varphi}\) is a supersolution to the adjoint heat equation
    \[
        \left( \frac{\partial}{\partial t} + \lp - R \right) (S e^{-\varphi}) = 2 e^{-\varphi} |\ric + D^2 \varphi|^2 \geq 0.
    \]
    The monotonicity of the \(\cF\)-functional follows from integrating this inequality over \(M\).
\end{remark}


\end{document}